\documentclass[11pt]{amsart}
\usepackage{eucal}
\usepackage{amssymb}
\usepackage{epsfig}
\usepackage{amscd, comment, color}

\usepackage{hyperref}

\newtheorem{thm}{Theorem}[section]
\newtheorem{Prop}[thm]{Proposition}
\newtheorem{lemma}[thm]{Lemma}

\newcommand{\bc}{{\mathbb C}}
\newcommand{\bC}{{\mathbb C}}

\newcommand{\bh}{{\mathbb H}}

\newcommand{\br}{{\mathbb R}}

\newcommand{\bZ}{{\mathbb Z}}

\newcommand{\ra}{\rightarrow}

\newcommand{\Z}{\mathbb{Z}}
\newcommand{\U}{\operatorname{U}}
\newcommand{\SU}{\operatorname{SU}}
\newcommand{\PU}{\operatorname{PU}}
\renewcommand{\U}{\operatorname{U}}
\newcommand{\Sp}{\operatorname{Sp}}
\newcommand{\SL}{\operatorname{SL}}
\newcommand{\Ad}{\operatorname{Ad}}

\newcommand{\GE}{\Gamma_\textrm{8}}
\newcommand{\GW}{\Gamma_\textrm{W}}

\newcommand{\Hom}{\operatorname{Hom}}

\newcommand{\id}{\operatorname{id}}

\let \cal \mathcal
\newenvironment{pf}{\begin{trivlist}\item[]{\bf Proof:\ }}
{\mbox{}\hfill\rule{.08in}{.08in}\end{trivlist}}

\begin{document}
\title[Quaternionic deformations of discrete groups]
{Deformation spaces of discrete groups of SU(2,1) in quaternionic hyperbolic plane:
a case study}
\author{Antonin Guilloux}
\address{Sorbonne Universit\'e, CNRS, IMJ-PRG and INRIA, OURAGAN, 4, Place Jussieu, 75005 Paris, France}
\email{antonin.guilloux@imj-prg.fr}

\author{Inkang Kim}
\address{School of Mathematics,
     KIAS, Hoegiro 85, Dongdaemun-gu,
     Seoul, 130-722, Korea}
\email{inkang@kias.re.kr}
     
\date{}
\begin{abstract}
In this note, we study deformations of discrete and Zariski dense subgroups 
of $\SU(2,1)$ in the isometry group $\Sp(2,1)$ of quaternionic hyperbolic space. Specifically we consider two examples coming from
 representations of $3$-manifold groups (the figure eight knot and Whitehead links complement)
 and show opposite behaviors: one is not deformable outside 
 $U(2,1)$, while the other
 has a big space of deformations in $\Sp(2,1)$.
\end{abstract}

\maketitle

 \footnotetext[1]{2000 {\sl{Mathematics Subject
Classification.}} 51M10, 57S25.} \footnotetext[2]{{\sl{Key words and
phrases.}} Quaternionic hyperbolic space, complex hyperbolic space,
local rigidity, representation variety.} \footnotetext[3]{The second author
gratefully acknowledges the partial support of the grant (NRF-2017R1A2A2A05001002) and a warm support of IHES during his stay.}

\section{Introduction}\label{intro}
In 1960's, A.  Weil  \cite{Weil} proved a local rigidity of a
uniform lattice $\Gamma\subset G$ inside $G$: he showed that
$H^1(\Gamma,\mathfrak{g})=0$ for any semisimple Lie group $G$ not
locally isomorphic to $\mathrm{SL}(2,\br)$. This result implies that the
canonical inclusion map $i:\Gamma \hookrightarrow G$ is locally
rigid up to conjugacy. In other words, for any local deformation
$\rho_t:\Gamma\ra G$ such that $\rho_0=i$, there exists a continuous
family $g_t\in G$ such that $\rho_t= g_t \rho_0 g_t^{-1}$. Weil's
idea is further explored by many others but notably by Raghunathan
\cite{Ra} and Matsushima-Murakami \cite{MM}. Much later Goldman and
Millson \cite{GM} considered the embedding of a uniform lattice
$\Gamma$ of $\mathrm{SU}(n,1)$
	$$\Gamma \hookrightarrow \mathrm{SU}(n,1)\hookrightarrow \mathrm{SU}(n+1,1)$$ and
proved that there is still a local rigidity inside $\mathrm{SU}(n+1,1)$ if one 
ignores a deformation coming from the center. More recently  further 
examples \cite{KKP, KP,KZ, K} of
local rigidity of a complex hyperbolic lattice in quaternionic
K\"ahler manifolds are described in the following situations:
$$\Gamma\hookrightarrow \mathrm{SU}(n,1)\subset \mathrm{Sp}(n,1)\subset \mathrm{SU}(2n,2)
\subset \mathrm{SO}(4n,4),$$
$$\Gamma\hookrightarrow \mathrm{SU}(n,1)\subset \mathrm{SU}(p,q),$$  $$\Gamma\hookrightarrow \mathrm{SU}(n,1)\subset \mathrm{Sp}(n+1, \mathbb R),$$$$ \Gamma\hookrightarrow \mathrm{SU}(n,1)\subset \mathrm{SO}(2n,2).$$

But all these examples deal with the standard inclusion map $\Gamma
\hookrightarrow G'$ to use the Weil's original idea about
$L^2$-group cohomology. We look in this paper at the more general
setting of a representation $\rho:\Gamma\ra G \subset G'$. We focus our attention to the case where
the representation is discrete and has Zariski-dense image in $G$. We seek
the possibility of deforming $\rho$ in $G'$ without being conjugate to a representation
landing in $G$.

In general, very little is known on this general problem. We study here deformations of two representations of non-uniform lattices of 
$\mathrm{SL}(2,\bC)$ inside $\Sp(2,1)$. Indeed, let $M_8$ be the figure eight knot complement
and denote by $\GE$ its fundamental group, and let $M_W$ be the Whitehead link
complement and $\GW$ its fundamental group. 

The character variety $\chi(\Gamma_8,\SU(2,1))$
is fully understood \cite{FGKRT} (see section \ref{Character} for the definition of character variety), and it contains $2$ (up to 
some equivalences) boundary unipotent
irreducible representations $\rho_0$ and $\rho_1$ which are already obtained in 
\cite{Falbel}, see also \cite{FalbelKoseleffRouillier}. We will  be mainly interested in 
the representation $\rho_0$ whose image is generated by the  following matrices in $\SU(2,1)$:
\[\left[ \begin{matrix}
     1 &   1   & \frac{-1-i\sqrt 3}{2}\\
     0 &   1   & -1  \\
     0 & 0 & 1 \end{matrix} \right]
     \quad \textrm{and} \quad \left[\begin{matrix}
   1 & 0 & 0 \\
   1 & 1 & 0\\
\frac{-1-i\sqrt 3}{2}&-1& 1 \end{matrix}\right]\]
In particular, we see that the image of $\rho_0$ is included in the Eisenstein-Picard 
arithmetic lattice of $\SU(2,1)$. It turns out that it is a \emph{thin} subgroup, as it is Zariski-dense.
We will show that $\rho_0$, as its surrounding lattice, is not deformable outside $\U(2,1)$. 
Recently some thin subgroups of finite index in $\Gamma_8$  were constructed inside lattices in $\SL(4,\br)$, that are indeed deformable (inside $\SL(4,\br)$) \cite{BD}.

Our knowledge of the character variety $\chi(\Gamma_W,\SU(2,1))$ is far less thorough. Boundary 
unipotent representations are described in \cite{FalbelKoseleffRouillier}, whereas a component
of this character variety has been described in \cite{GW}. We will consider a 
representation $\rho_W$ inside this component. Note that the image of $\rho_W$ is a free product
of two copies of $\bZ/3\bZ$ and is not contained in an arithmetic lattice. We will prove that
$\rho_W$ has a big space of deformations in $\Sp(2,1)$ and is therefore deformable outside  $\U(2,1)$.

We will first describe what is known about $\rho_0$ and $\rho_W$, exhibiting structural
differences. We then prove that the first one is rigid whereas the second one is deformable.
It would be very interesting to understand which properties of these representations lead to
the rigidity or deformability.

\section{Two opposite behaviors}

\subsection{Rigidity of $\rho_0$}

The fundamental group $\Gamma_8$ has a presentation \cite{FGKRT}:
$$\Gamma_8 = \langle a,b | b^{-1}aba^{-1}bab^{-1}a^{-1}ba^{-1} \rangle.$$
We consider the representation $\rho_0$ defined by the images of the generators:
$$\rho_0(a) = \left[ \begin{matrix}
     1 &   1   & \frac{-1-i\sqrt 3}{2}\\
     0 &   1   & -1  \\
     0 & 0 & 1 \end{matrix} \right] \quad \textrm{and} \quad   
     \rho_0(b) =\left[\begin{matrix}
   1 & 0 & 0 \\
   1 & 1 & 0\\
\frac{-1-i\sqrt 3}{2}&-1& 1 \end{matrix}\right]
$$
 
We prove in this paper that the representation $\rho_0$ cannot be deformed locally outside  $\U(2,1)$.
The proof is fairly straightforward, though involved computations are tedious. Here are the steps:
\begin{enumerate}
  \item As we will see in section \ref{CharVar-8}, at $[\rho_0]$,
  the character variety  $\chi(\Gamma_8,\U(2,1))$ is $3$-dimensional.
  \item We are able to compute the tangent space to $\chi(\Gamma_8,\Sp(2,1))$ at $[\rho_0]$: 
  it amounts to compute $H^1(\Gamma_8, \mathfrak{sp}(2,1)_{\mathrm{ad}(\rho_0)})$. 
  This homological computation will be explained in section \ref{hom-comp}. 
  The computed dimension is $3$.
  
  \item As we will recall in section \ref{generalities-character}, the natural map 
  $\chi(\Gamma_8,\U(2,1)) \to \chi(\Gamma_8,\Sp(2,1))$ is 
  a local diffeomorphism onto its image.
\end{enumerate}
Knowing these three facts, we see:
\begin{Prop}\label{prop:rigidity}
Every small deformation of $\rho_0 : \Gamma_8 \to \Sp(2,1)$ results in a representation conjugate to a representation $\Gamma_8 \to \U(2,1)$.
\end{Prop}

\subsection{Deformability of $\rho_W$}\label{deform}

Following \cite{GW}, the fundamental group $\Gamma_W$ has a presentation:
$$\Gamma_W = \langle a,b | aba^{-3}b^2a^{-1}b^{-1}a^3b^{-2} \rangle.$$
We consider the representation $\rho_W$ defined by the images of the generators:
$$\rho_W(a) = \left[ \begin{matrix}
     1 & \frac{\sqrt{3} - i\sqrt{5}}{2} & -1\\
\frac{-\sqrt{3} - i\sqrt{5}}{2} & -1 & 0\\
-1 & 0 & 0\end{matrix} \right]  \textrm{ and }  
     \rho_W(b) =\left[\begin{matrix}
   1 & -\frac{\sqrt{3} + i\sqrt{5}}{2} & -1\\
\frac{\sqrt{3} - i\sqrt{5}}{2} & -1 & 0\\
-1 & 0 & 0 \end{matrix}\right]
$$

Unlike the previous example, we prove in this paper that the representation $\rho_W$ can be deformed locally outside  $\U(2,1)$.
The proof is once again fairly straightforward. Here are the steps:
\begin{enumerate}
\item $\rho_W$ factors through a quotient $\Z_3* \Z_3$, and the whole component of the $\SU(2,1)$-character variety of $\Gamma_W$ does (see section \ref{CharVar-W}).
\item The $\U(2,1)$-character variety has dimension $6$ at $\rho_W$ (see section \ref{CharVar-W}).
\item The $\Sp(2,1)$-character of $\Z_3* \Z_3$ at $\rho_W$ has dimension at least $7$ (see section \ref{Wlocaldim}).
\end{enumerate}
We hence see that the $\Sp(2,1)$-character variety of $\Gamma_W$ has dimension, at $\rho_W$, at least $1$ more than the dimension of the $\U(2,1)$-character variety. It yields:
\begin{Prop}\label{prop:deformability}
There are small deformations of $\rho_W: \Gamma_W \to \Sp(2,1)$ which are
not conjugate to any representation $\Gamma_W \to \U(2,1)$.
\end{Prop}

\section{Character varieties}\label{Character}
The $G$-character variety of $\pi_1(M),$ denoted $\chi(\pi_1(M), G),$ is the geometric invariant theory quotient of $\Hom(\pi_1(M), G)$ by inner automorphisms of $G$. Often, some components of the character variety are realized as the space of $(G, X)$-structures on a given manifold $M$. Thurston studied
the Dehn surgery space of a hyperbolic knot complement in the early 70s using the idea of gluing tetrahedra
in hyperbolic 3-space. In his case, the variety appears as defined by his gluing equations \cite{Th}.
Thurston's approach is generalized to several different directions corresponding to different geometric structures such as spherical CR structure and real projective structure associated with Lie groups
$\text{SU}(2,1)$ and $\text{SL}(3,\mathbb R)$ respectively.  The latter one is known as a Hitchin component consisting of  convex real projective structures on a closed surface \cite{Hitchin}.

\subsection{General facts and definitions}\label{generalities-character}

For a given reductive algebraic group $G\subset \text{GL}(m,k)$ defined over $k$, and a finitely generated group $\Gamma$ with $n$-generators,
the {\bf representation variety} is $R(\Gamma, G)=\Hom(\Gamma, G)\subset G^n$, defined by the zero set of polynomials in $k[x_1,\cdots,x_{nm^2}]$. In this paper, $k=\mathbb R$ or $\mathbb C$. A representation
$\rho:\Gamma\ra G$ is {\bf Zariski dense} if the Zariski closure of the image is $G$. The group $G$ acts on $R(\Gamma, G)$ by conjugation, and it is well-known that the orbit of $\rho$ under conjugation
is closed if $\rho$ is Zariski dense.

Since the orbit under the conjugation is not closed in general, the quotient space of $R(\Gamma,G)$ under conjugation is not in general a Hausdorff space. To avoid this phenomenon, one takes the GIT quotient
$\chi(\Gamma,G)=R(\Gamma,G)// G$ to get again an algebraic set, called the {\bf character variety}.

In this paper,  all the representations we are considering  are not contained in $P\times Z(G)$ where $P$ is a  parabolic subgroup and $Z(G)$ is the center of $G$. In this case, the quotient  of $R(\Gamma,G)$ by the conjugation action of $G$ is nice around $\rho$ \cite[Section 1.3]{Gol}, and we can assure that 
the Zariski tangent space of the character variety at  $[\rho]$ can be computed by the first group cohomology of
$\Gamma$ with coefficient in $\mathfrak{g}_{Ad\rho}$.

We will need the following later.
\begin{lemma}\label{nonconjugate}
Let $\nu_1:\Gamma \ra \U(2,1)$ be a Zariski
dense representation which is  conjugate to $\nu_2:\Gamma \ra \U(2,1)$ in $\Sp(2,1)$. Then $\nu_1$ is conjugate  in
$\SU(2,1)$ to either $\nu_2$ or $\overline{\nu_2}$.
\end{lemma}

\begin{pf}Suppose $Q \nu_1 Q^{-1} =\nu_2$ for $Q\in \Sp(2,1)$. Since $\nu_1$ is Zariski dense,
$Q$ stabilizes $H^2_\bc$ inside $H^2_\bh$. If it is holomorphic, $Q\in \SU(2,1)$. Suppose
it is anti-holomorphic.  Any anti-holomorphic element  in $H^2_\bc$ can be written as $\iota$ followed by
an element in $U(2,1)$ where $\iota$ is a reflection along $H^2_\br$. By absorving the element in $U(2,1)$
we may assume that $Q$ restricted to $H^2_\bc$ is $\iota$. Now, $\iota$ can be realized as a complex conjugate $(z,w)\ra (\bar z,\bar w)$ in unit ball model.

Then $Q$ is realized by a diagonal matrix with entries $(j,j,j)$. Hence we get $Q\nu_1Q^{-1} = \overline{\nu_1}$ and $\nu_1 = \overline{\nu_2}$.

\end{pf}

\subsection{Description of the $\U(2,1)$-character variety for $\Gamma_8$}\label{CharVar-8}
Let $\Gamma_8$ denote the fundamental group of the figure eight knot complement in $\mathbb S^3$. 
Falbel-Guilloux-Koseleff-Rouillier-Thistlethwaite \cite{FGKRT} studied the character variety of $\Gamma_8$ in $\text{PGL}(3,\mathbb C)$ and $\PU(2,1)$. They describe a Zariski open set, through a variant of the character variety: the deformation variety. They show that there exist three irreducible components of the deformation variety. Each  one of these components is smooth of complex dimension two and contains a real-dimension $2$ subvariety of representations landing in $\PU(2,1)$ \cite[Section 5.3]{FGKRT}. 

\begin{Prop}\label{dimU}
The component of the $\U(2,1)$-character variety $\chi(\Gamma_8,\U(2,1))$ through $\rho_0$ has dimension $3$.
\end{Prop}

\begin{proof}
First of all, $\rho_0$ belongs to one of the components described in \cite{FGKRT}: it corresponds to the point $(u,v) = (-\sqrt{3}i,2)$ from \cite[Section 5.3]{FGKRT}. Hence, we know that the component of the $\SU(2,1)$-character variety $\chi(\Gamma_8,\SU(2,1))$ through $\rho_0$ has real dimension $2$.

Then $\Gamma_8$ is the fundamental group of a knot complement, so its abelianization is $\Z$. Hence the
character variety from $\Gamma_8$ to the center $\U(1)$ of $\U(2,1)$ is of real dimension $1$.

Now any representation $\Gamma_8 \to \U(2,1)$ can be locally decomposed as product of a representation in its center and a representation in $\SU(2,1)$. We get that the component of the $\U(2,1)$-character variety $\chi(\Gamma_8,\U(2,1))$ through $\rho_0$ has dimension $3$.
\end{proof}

\subsection{A known component of the $\U(2,1)$-character variety for $\Gamma_W$}\label{CharVar-W}

 Guilloux-Will studied the character variety $\chi(\Gamma_W, \text{SL}(3,\mathbb C))$. They showed that the
representations studied by Schwartz, Deraux, Falbel, Acosta, Parker, Will \cite{Sch, Der15,Der, Aco15,PW15} all belong to a common
algebraic component $X_0$ consisting of representations that factor through the group $\pi'=\mathbb{Z}_3*\mathbb{Z}_3$. Here $X_0$ is the character variety of $\pi'$ consisting of representations whose images are generated by two regular order 3 elements in $\text{SL}(3,\mathbb C)$. $X_0$ is of complex dimension 4, and the subset of representations in $\mathrm{SU}(2,1)$ is of real dimension $4$.

Moreover, the representation $\rho_W$ belongs to this component $X_0$ \cite[Section 3.4]{GW}. Using that the abelianization of $\Gamma_W$ is $\Z^2$, we get as before:

\begin{Prop}
The component of the $\U(2,1)$-character variety $\chi(\Gamma_W,\U(2,1))$ through $\rho_0$ has dimension $6$.
\end{Prop}

\section{Fox calculus and homological computations}\label{hom-comp}

\subsection{General presentation}

In this section, we briefly introduce a Fox calculus which is necessary for the calculation of the first
group cohomology and the Zariski tangent space of $\text{Hom}(\pi, G)$. For a detailed exposition, refer to \cite{Gol} Section 3. Such computations have already been used, e.g. in \cite{BA-H}.  Let $F_n$ be a free group on $n$-generators $x_1,\cdots,x_n$ and $\mathbb{Z}F_n$ the integral group ring. The augmentation homomorphism is a ring homomorphism
$$\epsilon:\mathbb{Z}F_n\rightarrow \mathbb{Z}$$ which maps an element $\sum_{\sigma\in F_n} m_\sigma \sigma$ to the coefficient sum $\sum_{\sigma\in F_n} m_\sigma $.  A derivation is a 
$\mathbb{Z}$-linear map $D:\mathbb{Z}F_n\rightarrow \mathbb{Z}F_n$ satifying $$D(m_1m_2)=D(m_1)\epsilon(m_2)+ m_1D(m_2).$$
Then the set of derivations $Der(F_n)$ is freely generated as a right $\mathbb{Z}F_n$-module by $n$ elements $\partial_i=\frac{\partial}{\partial x_i}$ which satisfy $\frac{\partial}{\partial x_i}(x_j)=\delta_{ij}$.
This derivation satisfies a useful rule of differential calculus, a mean value theorem,
$$u-\epsilon(u)=\sum (\partial_i u) (x_i-1)$$ for any $u\in \mathbb{Z}F_n$.

Let $\phi:F_n\rightarrow \mathrm{GL}(V)$ be a linear representation, which extends to a ring homomorphism
$\mathbb{Z}F_n\rightarrow \text{End}(V)$. Then a cocyle $u:\mathbb{Z}F_n\rightarrow V$ which
satisfies the cocycle identity $u(ab)=u(a)\epsilon(b)+\phi(a)u(b)$, can be written using the mean value
theorem as $$u(w)=\sum_{i=1}^n \phi(\partial_i w) u(x_i).$$

Using this Fox calculus, we can describe the Zariski tangent space to $\text{Hom}(\pi, G)\subset G^n$ for a group $\pi=F_n/\mathcal R$ where $\mathcal R$ is a normal subgroup of $F_n$ consisting of relations and $G$ is a Lie group whose Lie algebra is denoted by $\mathfrak g$.
Since an element in $\text{Hom}(\pi,G)$ corresponds to an element $\phi\in \text{Hom}(F_n,G)$ satisfying $\phi(R)=1$ for all $R\in \mathcal R$, the Zariski tangent space to $\text{Hom}(\pi,G)$ at
$\phi\in \text{Hom}(\pi,G)$ is the space of cocycles
$$Z^1(\pi, \mathfrak{g}_{\Ad \phi})=\{(u_1,\cdots,u_n)\in \mathfrak{g}^n| \sum_{i=1}^n \Ad \phi (\partial_i R)u_i=0,\ \text{for all}\ R\in \mathcal{R}       \}$$  by associating $(\mu(x_1),\cdots, \mu(x_n))$ to each 1-cocycle $\mu$.

Moreover, in order to have the Zariski tangent space to the character variety, you have to mod out by the coboundaries $B^1(\pi,\mathfrak{g}_{Ad \phi})$. In this setting, a coboundary is an element $(u_1,\ldots, u_n) \in \mathfrak{g}^n$ such that there exist some $u \in \mathfrak g$ with:
$$\forall 1\leq i \leq n, \quad u_i = \Ad \phi(x_i) u - u.$$

\subsection{Effective computations for $\Gamma_8$}

The material presented above can be tackled in a very concrete and effective manner. Let us describe the involved computations for the representation $\rho_0 : \Gamma_8 \to \Sp(2,1)$. The actual computations are basic linear algebra, but with matrices a bit too big to be fully displayed here. A Sage Notebook \cite{Notebook} is available showing the computations done by a computer algebra system.

First of all, we use Fox calculus on our presentation of $\Gamma_8$:
$$\Gamma_8 = \langle a,b | b^{-1}aba^{-1}bab^{-1}a^{-1}ba^{-1} \rangle.$$
Let us denote by $R$ the relation $b^{-1}aba^{-1}bab^{-1}a^{-1}ba^{-1}$. 
A straightforward computation gives:
\begin{itemize}
\item $\partial_a R = b^{-1} - b^{-1} a b a^{-1} + b^{-1} a b a^{-1} b - b^{-1} a b a^{-1} b a b^{-1} a^{-1} - b^{-1} a b a^{-1}b a b^{-1} a^{-1} b a^{-1}.$
\item $\partial_b R = -b^{-1} + b^{-1} a + b^{-1} a b a^{-1} - b^{-1} a b a^{-1} b a b^{-1} + b^{-1} a b a^{-1} b a b^{-1} a^{-1}$.
\end{itemize}
Let us note that in Sagemath, the Fox calculus is implemented and the result of this computation is given by the so-called Alexander matrix.

From this, the whole computation of the Zariski tangent space follows. This computation can be seen in the notebook, and the steps are: 
\begin{itemize}
\item Choose a basis for $\mathfrak {sp}(2,1)$: its cardinality is $21$. Pairs of vectors in this basis give a basis of the cochains $C^1$: as presented above, a cochain is seen as an element of $\mathfrak {sp}(2,1)^2$.
\item Compute both  $21\times 21$ matrices representing in this basis the adjoint action $\Ad(x)$ and $\Ad(y)$ of the generators, $x=\rho_0(a), y=\rho_0(b)$.
\item Compute $B^1$ as the image of the $42\times 21$ matrix $\begin{pmatrix} \Ad(x) - \id\\\Ad(y) - \id\end{pmatrix}$ in the chosen basis.
\item Using the different terms $\partial_a R$, $\partial_b R$ appearing in the definition of $Z^1$ as in the previous lemma,
applying $\Ad(\rho_0)$ to these expressions, we get the  $21\times 42$ matrix whose kernel is $Z^1$:
$$\begin{pmatrix}  \Ad\rho_0(\partial_a R), \Ad\rho_0(\partial_b R)  \end{pmatrix}\begin{pmatrix} u_1\\ u_2 \end{pmatrix}=0.$$
\item Compute the dimension of $Z^1/B^1$. Note that $\rho_0$ has entries in a number field: the computation can be done exactly and the computed dimension has a true meaning.
\end{itemize}

As a result of this computation, we get:
\begin{Prop}\label{dimSp}
The component of the character variety $\chi(\Gamma_8,\Sp(2,1))$ through $\rho_0$ has dimension 3.
\end{Prop}

Proof of Proposition \ref{prop:rigidity}: The Lie algebra $\mathfrak{sp}(2,1)$ decomposes as $\mathfrak{u}(2,1)\oplus S^2{\mathbb C}^3$ under $\rho_0$ as a real representation, see \cite{KKP}.  Hence
$$H^1(\Gamma_8, \mathfrak{sp}(2,1))=H^1(\Gamma_8, \mathfrak{u}(2,1))+ H^1(\Gamma_8, S^2{\mathbb C}^3).$$
By above Propositions \ref{dimU} and \ref{dimSp},  $H^1(\Gamma_8, S^2{\mathbb C}^3)=0$, which implies that all the small deformations of $\rho_0$ in $\Sp(2,1)$ are conjugate to the ones in $\U(2,1)$.

\section{Order 3 elements and the deformation of $\rho_W$}\label{Wlocaldim}

We compute in this section a lower bound on the dimension of a component of the character variety $\chi(\Gamma_W,\Sp(2,1))$:
\begin{Prop}
The dimension around $[\rho_W]$ of the character variety $\chi(\Gamma_W,\Sp(2,1))$ is at least $7$.
\end{Prop}

\begin{proof}
As we saw in Section \ref{deform}, the image of $\rho_W$ is isomorphic to $\Z_3 \star \Z_3$, with $\rho_W(a)$ and $\rho_W(b)$ being  two order $3$ generators.

Moreover, as recalled in section \ref{CharVar-W}, the whole component of the $\SU(2,1)$-character variety containing $[\rho_W]$ is made from representations $[\rho]$ with $\rho(a)$ and $\rho(b)$ being two order $3$ elements of $\SU(2,1)$.

Let $\mathcal E = \left\{ (A,B) \in \Sp(2,1)^2 \quad A^3=B^3=1\right\}$. Then we have an inclusion $\mathcal E/\Sp(2,1) \to \chi(\Gamma_W, \Sp(2,1))$.

As a matrix of $\SU(2,1)$, the eigenvalues of $\rho_W(\alpha)$ are $1$, $\omega$, $\omega^2$, 
where $\omega^3=1$ in $\bc$. So, inside $\Sp(2,1)$, $\rho_W(\alpha)$ is conjugate \cite{CaoGongopadhyay} to the matrix 
$A=\begin{pmatrix} 1 & 0 & 0 \\
                    0 & \omega & 0 \\
                    0 & 0 & \omega 
\end{pmatrix}$.

By deforming the pair $(\rho_W(\alpha),\rho_W(\beta))$ to a pair of order $3$ matrices in 
$\Sp(2,1)$ and up to conjugation, we may assume that the first one always equals $A$. 
Its centralizer \cite[Section 5.1]{CaoGongopadhyay} in $\Sp(2,1)$ is the subgroup of block-diagonal matrices:
$$Z = \left\{ \begin{pmatrix} x & \\ & X\end{pmatrix}\in \Sp(2,1) \quad\textrm{where }x\in \bh, X \in\U(2)\right\}.$$
Note that the dimension of $Z$ is $7$.

The second matrix $B$ of the pair is another order $3$ matrix, conjugate to $A$. So we are indeed looking at the set of pairs $(A,gAg^{-1})$ up to conjugation. In other terms, let 
$$\mathcal E' = \left\{ (A,gAg^{-1}) \textrm{ for } g\in \Sp(2,1)\right\}.$$ Then locally around $[\rho_W]$ we have $\mathcal E/\Sp(2,1) = \mathcal E'/Z$.

Eventually, we see that for any $g\in\Sp(2,1)$ and $h\in\Sp(2,1)$, the two pairs $(A,gAg^{-1})$ and 
$(A,hAh^{-1})$ are conjugate if and only if there exist $z_1$ and $z_2$ in $Z$ such that $h = z_1gz_2$. 

Hence the dimension of $\mathcal E'/Z$ is at least $\mathrm{dim}(\Sp(2,1)) -2\mathrm{dim}(Z) = 7$. This implies that the dimension around $[\rho_W]$ of $\chi(\Gamma_W,\Sp(2,1))$ is at least $7$.
\end{proof}
Indeed,  the dimension of $\chi(\Gamma_W,\Sp(2,1))$ around any point in the component $\cal C$ containing $[\rho_W]$ of $\chi(\Gamma_W,\U(2,1))$ 
is at least 7 as we can see as follows.
Note that any point in $\cal C$ can be written as a pair $(\alpha C, \beta B)$ with $\alpha,\beta\in U(1)$ and $C^3=B^3=I$ in $\SU(2,1)$. This point is conjugate to $(g_0\alpha g_0^{-1} A, g_0\beta g_0^{-1} h_0 A h_0^{-1})$  for some $g_0,h_0\in\Sp(2,1)$.  Now by varing $h\in \Sp(2,1)$, as in the proof above, there are at least 7-dimensional
space of $\{(g_0\alpha g_0^{-1} A, g_0\beta g_0^{-1} h A h^{-1})| h\in\Sp(2,1)\}$ near $(g_0\alpha g_0^{-1} A, g_0\beta g_0^{-1} h_0 A h_0^{-1})$ in $\chi(\Gamma_W, \Sp(2,1))$.

Note that the proposition \ref{prop:deformability} is now proven: the space of deformations of $\rho_W$ in $\Sp(2,1)$ has bigger dimension than the space of deformations in $U(2,1)$ showing that some deformations are not conjugate to $\U(2,1)$.

\bibliographystyle{plain}
\bibliography{biblio}

\end{document}